\documentclass{amsart}
\usepackage{amsmath}
\usepackage{amsthm}
\usepackage{amsfonts}
\usepackage{amssymb}
\usepackage{amsbsy}
\usepackage{amscd}
\usepackage{eucal}
\usepackage{amsaddr}
\usepackage{listings}

\newtheorem{theorem}{Theorem}

\numberwithin{equation}{section}

\newcommand{\beq}{\begin{equation}}
\newcommand{\eeq}{\end{equation}}

\newcommand{\skd}{\vspace*{0.2cm}}

\begin{document}
 
\title{New Approximation method for the computation of particular values of polylogarithms}

\author{Abdalla M. Aboarab}
\address{\sl Junior Student at Kafr El-Sheikh STEM School\\
E--mail: abdulah.10840@stemksheikh.moe.edu.eg}


\begin{abstract}
An efficient procedure for the computation of $Li_{s}(z)$ where $s<0$
is here presented. We started with Polylogarithm $Li_{s}(z)$ where $s<0$. The summation of $n^{s}z^{n}$ is evaluated using a new method. An assumption is made that the power $n$ is multiplied by $x$ where $x=1$; then the series is integrated $s$ time in order to cancel the term $n^s$ out leaving the term $z^n$. By simply taking the derivative of the result $s$-times, an expression to evaluate the series arises which include only a constant term and the s$^{th}$ order of derivative of the summation of a simple geometric series.
The computation of $Li_{s}(z)$ can then be performed in $\mathcal{O}(n)$ operations.
\end{abstract}

\maketitle

\section{Introduction}
\label{se:introduction}

The efficient computation of the Polylogarithm 
is a very important problem in numerical analysis and applied mathematics
with a wide range of applications including, just to mention a few,
in quantum statistics, the polylogarithm function appears as the closed form of integrals of the Fermi–-Dirac distribution and the Bose–-Einstein distribution, and is also known as the Fermi-–Dirac integral or the Bose-–Einstein integral\cite{DLMF} 

The difficulty of the problem lies essentially
in the fact that there is no formula introduced to reduce $Li_{s}(z)$. The only way to compute this function for large numbers is to sum up so many terms which takes much time to compute.

In this paper we present an alternative procedure.
The basic idea of the method is to introduce a general formula to reduce $Li_{s}(z)$ where $s<0$ by integrating and differentiating the series $s$-times with respect to an implemented variable $x$ where $x=1$, then we computed $Li_{s}(z)$ using the suggested method.

Finally, the numerical implementation of the algorithm follows straightforwardly and is very efficient when compared to the loop method. the s$^{th}$-order derivative at $x=1$ can be computed in $\mathcal{O}(n)$ operations.

\section{Derivation of the method}
\label{se:interpretation}

The standard form of the Polylogarithm reads \cite{wolf}:
\beq
Li_{s}(z) = \sum_{k=1}^\infty \frac{z^k}{k^s} 
\label{1}
\eeq
and when $s<0$ it turns into :
\beq
\sum_{k=1}^\infty k^s z^k 
\label{5bis}
\eeq
where ($k$ and $z \in \mathbb{R}$). We can now state the following theorem.
\begin{theorem}
\label{the:1}
Suppose there is an infinite sum $\sum_{k=1}^{\infty} k^{s}z^{k}$ where ($k$ and $z \in \mathbb{R}$). Then, this sum can be expressed as :
\beq
(\frac{1}{\ln z})^{s}\frac{d^{s}}{dx^{s}}(\frac{1}{1-z^{x}})\bigg|_{x=1}
\label{6}
\eeq

\end{theorem}

\begin{proof} 
we start with the finite sum $\sum_{k=1}^{b} k^{s}z^{k}$.
The power of $z$ can be multiply by $ x$ where $x = 1$. Then the sum becomes:
\beq
S = \sum_{k=1}^{b} k^{s}z^{kx}
\label{8}
\eeq
Integrating S with respect to $x$,

\beq
\int S dx=\frac{1}{\ln z}\sum_{k=1}^{b}k^{s-1}z^{kx}
\label{9}
\eeq
Then, by integrating S (s times), the s$^{th}$ integral of S becomes:
\beq
\int \int .....\int S (dx)^{s}=(\frac{1}{\ln z})^{s}\sum_{k=1}^{b}z^{kx}
\label{10}
\eeq
Now, by solving $\sum_{k=1}^{b}z^{kx}$as a geometric series,
\beq
\sum_{k=1}^{b}z^{kx}=\frac{z^x(z^{bx}-1)}{z^x-1}
\label{11}
\eeq
from equations (2.6) and (2.7), it is concluded that :
\beq
\int \int .....\int S (dx)^{s} = (\frac{1}{\ln z})^{s}(\frac{z^x(z^{bx}-1)}{z^x-1})
\label{12}
\eeq
Now, by applying the fundamental theorem of calculus, it is found that:
\beq
S = \sum_{k=1}^{b} k^{s}z^{kx} =  (\frac{1}{\ln z})^{s}\frac{d^s}{dx^s}(\frac{z^x(z^{bx}-1)}{z^x-1})\bigg|_{x=1}
\eeq
Now we go into the case that the series is evaluated between zero and
	infinity, To determine
	when $S$ converges, the series $S=\sum_{k=1}^{b} k^{s}z^{kx}$ is evaluated between zero and infinity. By testing convergence of the function by D’Alembert test, $\lim_{n \to \infty }(\frac{S(k+1)}{S(k)})$ has to be found. By solving the
	limit,
\beq
\lim_{k\to \infty} \left |(\frac{(k+1)^{s}.z^{(k+1)}}{k^{s}.z^{k}})  \right |=\lim_{k\to \infty}\left |((\frac{k+1}{k})^s .\frac{z^{(k+1)}}{z^{k}})  \right |=\lim_{K\to \infty}\left |((\frac{1+\frac{1}{k}}{1}) . z)  \right | = \left | z \right |
\eeq
And thus, three cases are to be considered. The first case is when $\left | z \right |=1$ then the sum turns into $\sum_{k=1}^{\infty} k^{s}$ which is Riemann zeta function $\zeta (-s)$ \cite{zeta}. The second case is when $\left | z \right |>1$; the series automatically
fails D’Alembert test, and consequently, it diverges. The third case is when $\left | z \right |<1$; it passes D’Alembert test because $\lim_{k \to \infty }\left |(\frac{S(k+1)}{S(k)})\right |<1$, and so, the series converges.
	
From the previous examination, it is concluded that the series converges if $\left | z \right |<1$ but it can be computed  for more values using analytic continuation\cite{wolf}. Applying Theorem 1 and equation (2.9) then taking the limit at infinity, the series can be written as:
\beq
\lim_{b\to \infty} (\frac{1}{\ln z})^{s}\frac{d^s}{dx^s}(\frac{z^x(z^{bx}-1)}{z^x-1})\bigg|_{x=1}=(\frac{1}{\ln z})^{s}\frac{d^s}{dx^s}(\frac{-z^x}{z^x-1})\bigg|_{x=1}
\eeq
and from equation (2.1) we conclude that:
\beq
Li_{-s}(z) = (\frac{1}{\ln z})^{s}\frac{d^s}{dx^s}(\frac{-z^x}{z^x-1})\bigg|_{x=1} \quad s>0,  \left | z \right |<1
\eeq
\end{proof}

\section{Computations}

The numerical implementation of the algorithm first requires the computation of the s$^{th}$ order derivative at $x=1$, then it multiplied by the term $(\frac{1}{\ln{z}})^s$ and it was computed using central finite difference \cite{comp} \cite{cent}.
\begin{align}
    \frac{d^nf}{dx^n}(x)=\frac{\delta^{n}_{h}[f](x)}{h^n}+\mathcal{O}(h^2)\quad (h<0)\\
    \delta^{n}_{h}[f](x)=\sum^n_{i=0}(-1)^i\binom{n}{i}f(x+(\frac{n}{2}-i)h)
\end{align}

Finally, formula (2.12) makes it possible to take full advantage of the computational efficiency of the 
the Fast numerical differentiation both in terms of speed of computation and of accuracy \cite{comp}. 
The calculation of the s$^{th}$ order derivative at $x=1$ can consequently be accomplished 
in $\mathcal{O}(n)$ operations.

Feasibility and accuracy of the algorithm have been verified by direct comparison 
of the obtained numerical results with the true Polylogarithm
values along with the absolute error are given
in Table \ref{tab:1}.
\begin{table}[p]
\caption{\label{tab:1} \it computed Polylogarithms $Li_s(z)$ for several random values.}

\begin{tabular}{|c|r|c|}    
\hline
\rule{0pt}{4ex} $\boldsymbol{s, z\,}$ 
& \bf Standard $\boldsymbol{Li_{s}(z)}$ \hspace{0.6cm} & \bf Computed $\boldsymbol{Li_{s}(z)}$ \\[+5pt]
\hline\rule{0pt}{3ex}
-1, 0.3  & 0.612244897959183673469 & 0.612245 \\[+5pt]
-2, 0.1  & 0.150891632373113854595 & 0.150892 \\[+5pt]
-2, 2    &-6.000000000000000000000 & -6.00001\\[+5pt] 
-3, 0.4   & 8.51851851851851851852 & 8.51759 \\[+5pt]
-4, 0.1  &0.374434791444393639181 & 0.374434 \\[+5pt]
-5, 0.56 & 3158.11837695681943777 & 3160.88 \\[+5pt]

\hline
\end{tabular} 
\end{table}

The increment of performances, with respect to the computation of the Polylogarithm, has been verified in terms of speed of computation at (nearly) equality of precision.
All accuracies have been determined by comparing the results of the algorithm with the reference values 
of $Li_{s}(z)$ computed with Mathematica \cite{Mathematica}.
All the results have confirmed increase of computational speed 
($\mathcal{O}(n)$).
Such an increase of performances will become even more crucial for numerical analysis and quantum statistics.

\skd
\skd

\appendix
\section{\\mathematica code for the method}

\begin{lstlisting}
z=(*the value of z*)
f[x_]=(-z^x)/(z^x-1)
n=(*order of derivative "s"*);
h=0.01;(*this number conrols the accuracy of approximation*)
delta = 0;
For[i=0;bin=1,i<=n,bin=bin*(n-i)/(i+1);i=i+1,
delta = delta  + ((-1)^i)*bin*f[1+h*(n/2-i)]
]
diff=delta/h^n;
(1/Log[z])^n(diff)




\end{lstlisting}

\end{document}